\documentclass[12pt]{article}
\usepackage{color,amsmath,amsfonts,amssymb,epsfig,latexsym}
\usepackage{graphicx}
\usepackage[latin1]{inputenc}
\usepackage{amssymb}
\usepackage{bm}
\usepackage{color}

\newcommand{\R}{\mbox{\F R}}

\def\R{{\mathbb R}}

\def\Q{{\mathbb Q}}
\def\P{{\mathbb P}}

\def\bu{{\bf u}}
\def\bU{{\bf U}}
\def\bq{{\bf q}}
\def\bx{{\bf x}}
\def\by{{\bf y}}

\def\ba{{\bf a}}
\def\bA{{\bf A}}

\def\by{{\bf y}}

\def\bomega{\boldsymbol\omega}
\def\bOmega{\boldsymbol\Omega}

\newtheorem{teo}{Theorem}[section]
\newtheorem{definition}{Definition}[section]
\newenvironment{proof}[1][Proof]{\medskip\noindent\textit{#1. }\upshape}{\medskip}
\newtheorem{lemma}{Lemma}[section]
\newtheorem{corollary}{Corollary}[section]

\numberwithin{equation}{section}
\begin{document}

\date{}
\title{Weak-strong uniqueness for fluid-rigid body interaction problem with
slip boundary condition}
\author{{\large Nikolai V. Chemetov$^1$, \v S\'arka Ne\v casov\'a$^2$, Boris
Muha $^3$} \\
%EndAName
\\
{\small $^1$ University of Lisbon }\\
{\small nvchemetov@gmail.com}\\
{\small $^2$ Institute of Mathematics, }\\
{\small \v Zitn\'a 25, 115 67 Praha 1, Czech Republic }\\
{\small matus@math.cas.cz}\\
{\small $^3$ Department of Mathematics}\\
{\small Faculty of Science}\\
{\small University of Zagreb, Croatia}\\
{\small borism@math.hr}}
\maketitle

\begin{abstract}
We consider a coupled PDE-ODE system describing the motion of the rigid body
in a container filled with the incompressible, viscous fluid. The fluid and
the rigid body are coupled via Navier's slip boundary condition. We prove
that the local in time strong solution is unique in the larger class of weak
solutions on the interval of its existence. This is the first weak-strong
uniqueness result in the area of fluid-structure interaction {with a moving boundary.}
\end{abstract}

\date{}

%\tableofcontents

%\title{Weak-strong uniqueness for fluid-rigid body interaction problem with slip boundary condition}

\section{Introduction}

%We consider problem from \cite{SarkaCemetov} and use the same notation.

\bigskip

\bigskip

A solid body motion in a fluid is a widespread phenomenon in nature, being
one of the most classical problem of fluid mechanics. The understanding of
the correct mathematical description of fluid-structure interaction has
several important applications in many branches, such as in civil
engineering, aerospace engineering, nuclear engineering, ocean engineering,
biomechanics and etc..

{
As well-known \cite{G2} the fluid motion fulfills the Navier-Stokes
equations and the solid motion is described by a system of ordinary
differential equations of momentum conservation laws. The fluid and solids are coupled through kinematic and dynamic coupling condition. The most usual kinematic coupling condition in the literature is the standard no-slip boundary condition: the
continuity of velocities of the fluid and the solids at the fluid-body
interfaces. Such approach has been investigated by many authors, e.g. \cite{CST}-%
\cite{FHN},\ \cite{HOST, SST} and references within. Despite lot of progress there are only few uniqueness results for weak solutions. The uniqueness of weak solution for the fluid-rigid body system in $2D$ was proven in \cite{GS15}. The uniqueness of weak solution to $3D$ Navier-Stokes equation is a famous open problem (see e.g. \cite{Galdi}). However, it is well known that strong solution (i.e. solution with some extra regularity or integrability) is unique in a larger class of weak solution on interval of its existence \cite{Serrin62}. These type of results are called weak-strong uniqueness results. The main result of this paper is to prove a weak-strong uniqueness result for the fluid-rigid body system. We also mention that in \cite{Disser2016} the authors studied the motion of rigid body containing a cavity filled with the fluid and proved a weak-strong uniqueness result for that problem.
}

However, it has been shown in \cite{HES},
\cite{HIL}, \cite{S} that the non-slip condition exhibits an unrealistic
phenomenon: two smooth solids can not touch each other.\ The non-slip
condition prescribes the adherence of fluid particles to the solid
boundaries and, as a consequence of a regularity of the fluid velocity,
permits the creation of fine boundary layer that does not allow the contact
of the solids.

Another method for coupling of the fluid and of the bodies admits the
slippage of fluid particles at the boundaries, which is described by
Navier's boundary condition. The first step in this direction of the study
of Navier's condition was done by Neustupa, Penel \cite{NP1}, \cite{NP2},
who demonstrate that the collision with a wall can occur for a prescribed
movement of a solid ball, when the slippage was allowed on both boundaries.
We refer for a discussion of Navier's boundary condition to Introduction of
\cite{muha}. In this last work a local in time existence result was
demonstrated for the motion of the fluid and an elastic structure with prescribed
Navier's condition on the boundaries.
{Motivated by these works, in this paper we study coupling via Navier's slip boundary condition.} Also we mention the article \cite{GH2}
where a local existence up to collisions of a weak solution for a
fluid-solid structure was proved.\ The existence of a strong solution in $2D$ case was
proven by Wang \cite{Wa}. \

The global in time existence of the weak solution was proven in \cite{CN} for a mixed case, when
Navier's condition was given on the solid boundary and the
non-slip condition on the domain boundary.  This result admits the collisions of
the solid with the domain boundary ({for more detailed discussion concerning influence of boundary conditions on the collision see \cite{GHC}}). Recently the local in time existence of
the strong solution for the mixed case was demonstrated in \cite{ACMN}.\

\section{Preliminaries}

We shall investigate the motion of a rigid body inside of a viscous
incompressible fluid. The fluid and the body occupy a bounded open domain $%
\Omega \subset \mathbb{R}^{N}$ $(N=2$ or $3)$. Let the body be a connected
open set ${S}_{0}\subset \Omega $ at the initial time $t=0.$\ The fluid
fills the domain $F_{0}=\Omega \backslash \overline{S_{0}}$ at $t=0.$

The Cartesian coordinates $\mathbf{y}$ of points of the body at $t=0$ are
called the Lagrangian coordinates. The motion of any material point $\mathbf{%
y}=(y_{1},..,y_{N})^{T}\in S_{0}$ is described by two functions:
\begin{equation*}
t\rightarrow \mathbf{q}(t)\in \mathbb{R}^{N}\quad \ \text{and }\quad
t\mapsto \mathbb{Q}(t)\in SO(N)\qquad \text{for}\quad t\in \lbrack 0,T],
\end{equation*}%
where $\mathbf{q}=\mathbf{q}(t)$ \ is the position of the body mass center
at a time $t$ and $SO(N)$ is the rotation group in $\mathbb{R}^{N},$ i.e.
the $\mathbb{Q}=\mathbb{Q}(t)$ \ is a matrix, satisfying $\mathbb{Q}(t)%
\mathbb{Q}(t)^{T}=\mathbb{I},$ $\mathbb{Q}(0)=\mathbb{I}$\ with $\mathbb{I}$
being the identity matrix. Therefore, the trajectories of all points of the
body are described by a preserving orientation isometry:%
\begin{equation}
{\mathbf{B}}(t,\mathbf{y})=\mathbf{q}(t)+\mathbb{Q}(t)(\mathbf{y}-\mathbf{q}%
(0))\qquad \text{ for any}\quad \mathbf{y}\in S_{0}  \label{is}
\end{equation}%
and the body occupies the set:%
\begin{equation}
S(t)=\{\mathbf{x}\in \mathbb{R}^{N}:\ \,\mathbf{x}=\mathbf{B}(t,\mathbf{y}%
),\quad \mathbf{y}\in S_{0}\}=\mathbf{B}(t,S_{0})  \label{set}
\end{equation}%
at any time $t.$ The velocity of the body, called \textit{\ \textit{rigid }
velocity}, is defined as:
\begin{equation}
\frac{d}{dt}{{\mathbf{B}}(t,\mathbf{y})}=\mathbf{u}_{s}(t,\mathbf{x})=\mathbf{a%
}(t)+\mathbb{P}(t)(\mathbf{x}-\mathbf{q}(t))\qquad \text{for all}\quad
\mathbf{x}\in S(t),  \label{comp}
\end{equation}%
where $\mathbf{a}=\mathbf{a}(t)\in \mathbb{R}^{N}$\ \ is the translation
velocity and $\mathbb{P}=\mathbb{P}(t)$ is the angular velocity. \ The
velocity $\mathbf{u}_{s}$ has to be compatible with $\mathbf{B}$ in the
sense:
\begin{equation}
\frac{d\mathbf{q}}{dt}=\mathbf{a}\quad \text{and}\quad \frac{d\mathbb{Q}}{dt}%
\mathbb{Q}^{T}=\mathbb{P}\quad \text{in}\;[0,T].  \label{comp1}
\end{equation}%
The angular velocity $\mathbb{P}$ is a skew--symmetric matrix, i.e. there
exists a vector $\bm{\omega}=\bm{\omega}(t)\in \mathbb{R}^{N},\mathbb{\ \ }$%
such that
\begin{equation}
\mathbb{P}(t)\mathbf{x}=\bm{\omega}(t)\times \mathbf{x},\qquad \forall
\mathbf{x}\in \mathbb{R}^{N}.  \label{om}
\end{equation}

%Let $S$, $\Omega\subset\R^3$ be bounded smooth domains such that $\overline{S}\subset{\rm Int}\Omega$, i.e. %dist$(S,\partial\Omega)>0$. We define the initial fluid domain $\Omega_F=\Omega\setminus \overline{S}$.

%The position of rigid body is determined by two functions
%$\bq:[0,T)\to \R^3$ and $\Q:[0,T)\to SO(3)$, where $\bq(t)$ is center of mass of rigid body $S$ at time $t$ and $\Q(t)$ %its rotation matrix. Therefore, the trajectories of all
%points of the body are described by a preserving orientation isometry
%$$
%{\bf A}(t,\by)=\bq(t)+\Q\big (\by-\bq(0)\big ),\; \by\in S,\;t\in [0,T).
%$$
%and the body occupies the set
%$$
%S(t) = {\bf A}(t, S),\; t\in [0,T).
%$$
%The rigid body velocity is given by the following expression:
%$$
%\bu_S(t,\bx)=\ba(t)+\P\big (\bx-\bq(t)\big ),\; \bx\in S(t),\; t\in [0,T),
%$$
%where $\ba(t)=\bq'(t)$, $\P(t)=\Q'(t)\Q(t)^\tau$ are translation and rotational velocity, respectively. Since, $\P$ is %skew-symmetric we have:
%$$
%\P(t)\bx=\bomega (t)\times\bx,\;\bx\in\R^3.
%$$
\noindent We define the fluid domain as $\Omega _{F}(t)=\Omega \setminus
\overline{S(t)}$.
%\ For simplicity {\color{red}we assume that the fluid density $\varrho $ and the mass of rigid body $\varrho _s$ equals $1$.}
{Finally, we introduce a slight abuse of notation to denote the non-cylindrical domains:
$$
(0,T)\times \Omega_F(t)=\bigcup_{t\in (0,T)}\{t\}\times\Omega_F(t), \quad
(0,T)\times \partial S(t)=\bigcup_{t\in (0,T)}\{t\}\times\partial S(t).
$$}

We consider the following problem
modeling the motion of the rigid body in viscous incompressible fluids.
\newline
Find $(\mathbf{u},p,\mathbf{q},{\mathbb{Q}})$ such that
\begin{equation}
\left.
\begin{array}{l}
{ \varrho} (\partial _{t}\mathbf{u}+(\mathbf{u}\cdot \nabla )\mathbf{u})=\mbox {div}%
\left( {\mathbb{T}}(\mathbf{u},p)\right) , \\
\mbox {div }\mathbf{u}=0%
\end{array}%
\right\} \;\mathrm{in}\; {(0,T)\times\Omega _{F}(t)},  \label{NS}
\end{equation}%
\begin{equation}
\left.
\begin{array}{l}
\frac{d^{2}}{dt^{2}}\mathbf{q}=-\int_{\partial S(t)}{\mathbb{T}}(\mathbf{u}%
,p)\mathbf{n}\,d\mathbf{\gamma }(\mathbf{x}), \\
\frac{d}{dt}({\mathbb{J}}\boldsymbol{\omega })=-\int_{\partial S(t)}(\mathbf{x}-\mathbf{%
q}(t))\times {\mathbb{T}}(\mathbf{u},p)\mathbf{n}\,d\mathbf{\gamma }(\mathbf{%
x})%
\end{array}%
\right\} \;\mathrm{in}\;(0,T),  \label{RigidBody}
\end{equation}%
\begin{equation}
(\mathbf{u}-\mathbf{u}_{s})\cdot \mathbf{n}=0,\qquad \beta (\mathbf{u}_{s}-%
\mathbf{u})\cdot \boldsymbol{\tau }={\mathbb{T}}(\mathbf{u},p)\mathbf{n}%
\cdot \boldsymbol{\tau }\qquad \mathrm{on}\;{(0,T)\times\partial S(t)},  \label{coupling}
\end{equation}%
\begin{equation}
\mathbf{u}(0,.)=\mathbf{u}_{0}\qquad \mathrm{in}\;\Omega ;\qquad \mathbf{q}%
(0)=\mathbf{q}_{0},\qquad \mathbf{q}^{\prime }(0)=\mathbf{a}_{0},\qquad
\boldsymbol{\omega }(0)=\boldsymbol{\omega }_{0},  \label{IC}
\end{equation}%
where $\mathbf{n}($\textbf{$x$}$)$ is the unit \textit{interior} normal at \
$\mathbf{x}\in \partial S(t),$ i.e. the vector $\mathbf{n}$ is directed
inside of $S(t)$, {$\varrho$ is a constant density of fluid},  { and $\beta>0$ a constant friction coefficient of $\partial S$}. The surface measure over a moving surface $\partial S(t)$
is indicated by $d\mathbf{\gamma }.$ $\mathbb{J}$ is the matrix of the
inertia moments of the body $S(t)$ related to its mass center, calculated as:
\begin{equation*}
\mathbb{J}=\int_{S(t)}{\varrho_S(t,\bx)} (|\mathbf{x}-\mathbf{q}(t)|^{2}\mathbb{I}-(\mathbf{x}-%
\mathbf{q}(t))\otimes (\mathbf{x}-\mathbf{q}(t)))\,d\mathbf{x},
\end{equation*}
where $\varrho_S$ is the rigid body density.
In \eqref{NS} $\mathbf{u}$ is the fluid velocity; ${\mathbb{T}}$ is the
stress tensor and $\mathbb{D}$ is the deformation-rate tensor, which are
defined as:
\begin{equation*}
{\mathbb{T}}=-pI+2\mu \,\mathbb{D}\mathbf{u}\qquad \text{and}\qquad \mathbb{D%
}\mathbf{u}=\frac{1}{2}\left( \nabla \mathbf{u}+\left( \nabla \mathbf{u}%
\right) ^{T}\right)
\end{equation*}%
with $p$ being the fluid pressure and $\mu >0$ being the constant viscosity
of the fluid.
For simplicity {we assume that the fluid density $\varrho $ and the mass of rigid body equals $1$.}
\bigskip

Let us introduce the definition of weak solutions for system \eqref{NS}-%
\eqref{IC}. To begin with we define the space \cite{lions,Ne}:%
\begin{equation*}
V^{0,2}(\Omega )=\{\mathbf{v}\in L^{2}(\Omega ):\ \,\mbox{div }\mathbf{v}%
=0\;\ \text{ in}\;\mathcal{D}^{\prime }(\Omega ),\qquad \mathbf{v}\cdot
\mathbf{n}=0\;\ \text{ in}\;H^{-1/2}(\partial \Omega )\},
\end{equation*}%
where $\mathbf{n}$ is the unit normal to the boundary of $\Omega .$ \ Let $%
\mathcal{M}(\Omega )$ be the space of bounded Radon measures. Let%
\begin{equation*}
BD_{0}(\Omega )=\left\{ \mathbf{v}\in L^{1}(\Omega ):\ \,\mathbb{D}\mathbf{v}%
\in \mathcal{M}(\Omega ),\qquad \,\mathbf{v}=0\;\ \ \mbox{   on  }\;\partial
\Omega \right\}
\end{equation*}%
be the space of functions of bounded deformation. Let $S$\ be an open
connected subset of $\Omega $ with the boundary $\partial S\in C^{2}$. We
introduce the following space of vector functions:
\begin{eqnarray*}
KB(S) &=&\left\{ \mathbf{v}\in BD_{0}(\Omega ):\,\ \mathbb{D}\mathbf{v}\in
L^{2}(\Omega \backslash \overline{S}),\qquad \mathbb{D}\mathbf{v}=0\ \ \text{%
a.e. on }S,\right. \\
&&\left. \qquad \quad \quad \qquad \;\;\;\mbox{div}\mathbf{v}=0\ \
\mbox{
in }\ \mathcal{D}^{\prime }(\Omega )\right\} .
\end{eqnarray*}

\bigskip

Now, we can give a definition of the weak solutions of \eqref{NS}-%
\eqref{IC}.

\begin{definition}
\label{definition} The triple $\left\{ \mathbf{B},\mathbf{u}\right\} $ is a
weak solution of system {\eqref{NS}-\eqref{IC}}, if the following conditions
are satisfied:

1) The function $\mathbf{B}(t,\cdot ):\mathbb{R}^{N}\rightarrow \mathbb{R}%
^{N}$ \ is a preserving orientation isometry \eqref{is}, which defines a
time dependent set $S(t)$ by \eqref{set}. The isometry $\mathbf{B}$ is
compatible with $\mathbf{u}=\mathbf{u}_{s}$ on $S(t)$: \ the functions $%
\mathbf{q},$ $\mathbb{Q}$ are absolutely continuous on $[0,T]$\ and satisfy
equalities \eqref{comp}-\eqref{om};

%2) The function $\rho \in L^{\infty }(Q)$ satisfies the integral equality%
%\begin{equation}
%\int_{Q}\rho (\xi _{t}+(\mathbf{u}\cdot \nabla )\xi )\,dtd\mathbf{x}%
%=-\int_{\Omega }\rho _{0}\xi (0,\cdot )\,d\mathbf{x}  \label{eq5}
%\end{equation}%
%for\ any $\xi \in C^{1}(\overline{Q})$, $\ \xi (T,\cdot )=0;$

2) The function $\mathbf{u}\in L^{2}(0,T;KB(S(t)))\cap L^{\infty
}(0,T;V^{0,2}(\Omega ))$ satisfies the integral equality:
\begin{align}
\int_{0}^{T}dt\int_{\Omega \backslash \partial S(t)}\{\mathbf{u}\cdot{\partial_t}\boldsymbol{%
\psi }&+(\mathbf{u}\otimes \mathbf{u}) :\mathbb{D}\boldsymbol{\psi }%
-2\mu _{f}\,\mathbb{D}\mathbf{u}:\mathbb{D}\boldsymbol{\psi }\,\}d\mathbf{x}
\notag \\
& =-\int_{\Omega }\mathbf{u}_{0}\cdot\boldsymbol{\psi }(0,\cdot )\,d\mathbf{x}%
+\int_{0}^{T}dt\int_{\partial S(t)}\beta (\mathbf{u}_{s}-\mathbf{u}_{f})\cdot(%
\boldsymbol{\psi }_{s}-\boldsymbol{\psi }_{f})\,d\mathbf{\gamma },
\label{Weak}
\end{align}%
which holds for any test function $\boldsymbol{\psi },$ such that%
\begin{eqnarray}
\boldsymbol{\psi } &\in &L^{2(N-1)}(0,T;KB(S(t))),  \notag \\
{\partial_t}\boldsymbol{\psi } &\in &L^{2}(0,T;L^{2}(\Omega \backslash \partial
S(t))),\qquad \boldsymbol{\psi }(T,\cdot )=0.  \label{test}
\end{eqnarray}%
By $\mathbf{u}_{s}(t,\mathbf{\cdot }),$ $\boldsymbol{\psi }_{s}(t,\mathbf{%
\cdot })$ and $\mathbf{u}_{f}(t,\mathbf{\cdot }),$ $\boldsymbol{\psi }_{f}(t,%
\mathbf{\cdot })$\ \ we denote the trace values of $\mathbf{u},$ $%
\boldsymbol{\psi }$ \ on $\partial S(t)$\ \ from \ the "\textit{rigid}" side
$S(t)$ and the "fluid" side $F(t)$, respectively.
\end{definition}

\bigskip

Let us recall the global solvability result\ proved in \cite{CN}.

\begin{teo}
\label{theorem} Let the boundaries be $\partial \Omega \in C^{0,1}$, $%
\partial {S}_{0}\in C^{2}$. Let us assume that $\overline{S_{0}}\subset
\Omega $ and
\begin{equation*}
\mathbf{u}_{0}\in V^{0,2}(\Omega ),\qquad \mathbb{D}\mathbf{u}_{0}=0\ \quad
\text{in \ }\mathcal{D}^{\prime }(S_{0}).
\end{equation*}%
Then problem {\eqref{NS}-\eqref{IC}}\ \ possesses a weak solution $\left\{
\mathbf{B},\mathbf{u}\right\} ,$ such that the isometry $\mathbf{B}(t,\cdot
) $ is Lipschitz continuous with respect to $t\in \lbrack 0,T],$ \ \
\begin{equation*}
\mathbf{u}\in C_{\mathrm{weak}}(0,T;V^{0,2}(\Omega ))\cap L^{2}(0,T;KB(S(t)))
\end{equation*}%
and for a.a. $t\in (0,T)$ the following inequality holds:
\begin{equation}
\frac{1}{2}\int_{\Omega }|\mathbf{u}|^{2}(t)\ d\mathbf{x}+\int_{0}^{t}dr%
\left\{ \int_{\Omega _{F}({r})}2\mu \,|\mathbb{D}\,\mathbf{u}|^{2}\,\,d\mathbf{%
x}+\int_{\partial S(r)}\beta |\mathbf{u}-\mathbf{u}_{s}|^{2}\ d\mathbf{%
\gamma }\right\} \leq \frac{1}{2}\int_{\Omega }|\mathbf{u}_{0}|^{2}\ d%
\mathbf{x}.  \notag
\end{equation}
\end{teo}

\bigskip

\bigskip

Let us introduce on the fluid domain $\Omega _{F}(t)$ for $t\in (0,T)$ the
following function spaces:%
\begin{equation*}
L^{2}(\Omega _{F}^{1}(t)),\qquad L^{\infty }(\Omega _{F}^{1}(t)),\qquad
W^{1,\infty }(\Omega _{F}^{1}(t))\qquad \text{for a.e.\ }t\in (0,T)
\end{equation*}%
and%
\begin{equation*}
L^{2}(0,T;L^{2}(\Omega _{F}(t))),\qquad L^{2}(0,T;H^{k}(\Omega
_{F}(t))),\qquad H^{k}(0,T;L^{2}(\Omega _{F}(t)))\quad \text{with \ }k=1%
\text{ or }2.
\end{equation*}

\bigskip

Also we recall the local existence result for the strong solution obtained
in \cite{ACMN}.

\begin{teo}
\label{theorem-strong} Let the boundaries be $\partial \Omega ,$ $\partial {S%
}_{0}\in C^{2}$. Suppose that $\overline{S_{0}}\subset \Omega  $ {  and}
\begin{equation*}
\mathbf{u}_{0}\in KB(S_{0}).
\end{equation*}%
Then problem {\eqref{NS}-\eqref{IC}}\ possesses a local-in-time strong
solution $(\mathbf{B},\mathbf{u})$, such that the isometry $\mathbf{B}%
(t,\cdot )$ is Lipschitz continuous with respect to $t\in \lbrack 0,T],$ and
there exists a maximal $T_{0}>0$, such that {\eqref{NS}-\eqref{IC}} has a
unique strong solution $(\mathbf{u},p,\mathbf{a}(t),\boldsymbol{\omega }(t))$,
which for all $T<T_{0}$ satisfies the following energy inequality:
\begin{equation*}
\Vert \mathbf{u}\Vert _{L^{2}(0,T;H^{2}(\Omega _{F}(t))}+\Vert p\Vert
_{L^{2}(0,T;H^{1}(\Omega _{F}(t))}+\Vert \mathbf{a}\Vert _{H^{1}(0,T)}+\Vert
\boldsymbol{\omega }\Vert _{H^{1}(0,T)}\leq C.
\end{equation*}%
Here and below we denote by $C$ generic constants depending only on the data
of our problem {\eqref{NS}-\eqref{IC}.}
\end{teo}

\bigskip

The aim of our paper is to prove the weak-strong uniqueness result for
system \eqref{NS}-\eqref{IC}. More precisely, we prove that on the interval $%
(0,T_{0})$ where the strong solution exists, the strong solution is unique
in the class of weak solutions given by Definition \ref{definition}. To the
best of our knowledge this is the first weak-strong uniqueness in the area of
fluid-structure interaction. Uniqueness of weak solutions in $2D$ case for
the fluid-rigid body system with no-slip coupling condition was proved in
\cite{GS15}. %\begin{teo}
%\label{theorem-strong}  Let the boundaries be $\partial
%\Omega \in C^{0,1}$, $\partial {S}_{0}\in C^{2}$. Let us assume that $%
%S_{0}\subset \Omega $ such that $dist[S_{0},\Omega ]>0,$ and let
%\begin{equation}
%\qquad \mathbf{u}_{0}\in V^{0,2}(\Omega )\cap H^1(\Omega _F),\quad \,\mathbb{D}\mathbf{u}%
%_{0}=0\quad \text{in \ }\mathcal{D}^{\prime }(S_{0}).  \label{data}
%\end{equation}%
%Then problem {\eqref{NS}-\eqref{coupling}}\ \ possesses a local strong solution $%
%\left\{ \mathbf{B},\mathbf{u}\right\} ,$ such that the isometry $%
%\mathbf{B}(t,\cdot )$ is Lipschitz continuous with respect to $t\in \lbrack
%0,T],$%

Let us end this section by a well-known\textbf{\ }\textit{Reynolds transport
theorem} from the fluid mechanics theory, which will be often used in our
calculations on moving domains.

\begin{lemma}
\label{Rey} Let $V(t)$ be a time dependent volume moved by a smooth {divergence-free} velocity
$\mathbf{v}=\mathbf{v}(t,\mathbf{x})$. Then
\begin{equation}
\frac{d}{dt}\int_{V(t)}f(t,\mathbf{x})\,dx=\int_{V(t)}\dot{f}\,d%
\mathbf{x}  \label{eq6}
\end{equation}%
for any smooth function $f=f(t,\mathbf{x})$. Here $\dot{f}=\frac{%
\partial f}{\partial t}+(\mathbf{v}\cdot \nabla )f$ is the total time
derivative.
\end{lemma}

\bigskip

%$\mathbf{u}\in C_{\mathrm{weak}}(0,T;V^{1,2}(\Omega )) \cap {L^2(0,T;H^2(\Omega_F(t))}$, $p \in %L^2(0,T;H^1(\Omega_F(t))$, $\ba\in H^1(0,T)$,   $\bomega \in H^(0,T)$ and the energy
%inequality%
%\begin{equation}\label{reg1}
%\|{\bf u}\|_{L^2(0,T;H^2(\Omega_F(t))}+\|p\|_{L^2(0,T;H^1(\Omega_F(t))}+\|\ba\|_{H^1(0,T)}+\|\bomega\|_{H^(0,T)}\leq %C({\rm data})
%\end{equation}
%is satisfied.
%\end{teo}

\section{Weak-strong uniqueness}

Let $(\mathbf{u}_{1},\mathbf{a}_{1},\boldsymbol{\omega }_{1})$ be the
triplet consisting of the fluid velocity, the translation rigid body
velocity and the angular rigid body velocity connected to the { {\it weak solution}} $%
(\mathbf{u}_{1},\mathbf{B}_{1})$ (see Definition \ref{definition}), i.e.
\begin{equation*}
\mathbf{B}_{1}(t,\mathbf{y})=\mathbf{q}_{1}(t)+{\mathbb{Q}}_{1}(t)(\mathbf{y}%
-\mathbf{q}_{0}),
\end{equation*}%
where $\mathbf{a}_{1}=\mathbf{q}_{1}^{\prime }(t)\in \mathbb{R}^{3}$ and {  the vector } $%
\boldsymbol{\omega }_{1}=\boldsymbol{\omega }_{1}(t)\in \mathbb{R}^{3}$ \ is
associated with the skew--symmetric matrix $\mathbb{P}_{1}={\mathbb{Q}}%
_{1}^{\prime }{\mathbb{Q}}_{1}^{T}$, satisfying the property:
\begin{equation}
\boldsymbol{\omega }_{1}(t)\times \mathbf{x}=\mathbb{P}_{1}(t)\mathbf{x}%
,\qquad \forall \mathbf{x}\in \mathbb{R}^{3}.  \label{1}
\end{equation}%
{ We denote the domain of the rigid body at the time $t$ by
\begin{equation*}
S_{1}(t)=\mathbf{B}_{1}(t,S_{0})
\end{equation*}%
 and the corresponding fluid domain by
\begin{equation*}
\Omega _{F}^{1}(t)=\Omega \setminus \overline{S_{1}(t)}.
\end{equation*}
}

Moreover, let $(\mathbf{u}_{2},\mathbf{a}_{2},\boldsymbol{\omega }_{2})$ be
the { {\it strong solution} } given by Theorem \ref{theorem-strong} with the
corresponding rigid deformation $\mathbf{B}_{2}$:
\begin{equation*}
\mathbf{B}_{2}(t,\mathbf{y})=\mathbf{q}_{2}(t)+{\mathbb{Q}}_{2}(t)(\mathbf{y}%
-\mathbf{q}_{0}),
\end{equation*}%
where $\mathbf{a}_{2}=\mathbf{q}_{2}^{\prime }$ and
{  the vector } $\boldsymbol{\omega }%
_{2} $ is associated with the skew--symmetric matrix $\mathbb{P}_{2}={%
\mathbb{Q}}_{2}^{\prime }{\mathbb{Q}}_{2}^{T}$, satisfying
\begin{equation}
\boldsymbol{\omega }_{2}(t)\times \mathbf{x}=\mathbb{P}_{2}(t)\mathbf{x}%
,\qquad \forall \mathbf{x}\in \mathbb{R}^{3}.  \label{2}
\end{equation}%
Also as before, { we denote the domain of the rigid body at the time $t$ by
\begin{equation*}
S_{2}(t)=\mathbf{B}_{2}(t,S_{0})
\end{equation*}%
 and the corresponding fluid domain by
\begin{equation*}
\Omega _{F}^{2}(t)=\Omega \setminus \overline{S_{2}(t)}.
\end{equation*}%
}

In this article our main objective is to demonstrate the following
weak-strong uniqueness theorem.

\begin{teo}
\label{WeakStrong} We will prove the weak-strong uniqueness result, i.e.
\begin{equation*}
(\mathbf{u}_{1},\mathbf{a}_{1},\boldsymbol{\omega }_{1})=(\mathbf{u}_{2},%
\mathbf{a}_{2},\boldsymbol{\omega }_{2})\qquad \text{on the time interval }%
(0,T_{0}),
\end{equation*}%
where the strong solution $(\mathbf{u}_{2},\mathbf{a}_{2},\boldsymbol{\omega
}_{2})$ exists.
\end{teo}

\bigskip

The demonstration of this theorem we divide on few steps, proving auxiliary
Lemmas \ref{Strong1}-\ref{weak_strong}. The major difficulty in the study of
this uniqueness result consists from the fact that the fluid domains of $%
\mathbf{u}_{1}$ and $\mathbf{u}_{2}$ are a priori different and therefore we
need to transform the strong solution to the fluid domain of the weak
solution in order to compare them.

Let $\mathbf{X}_{1}$ and $\mathbf{X}%
_{2} $ be two time dependent changes of variables defined in Appendix \ref%
{Sec:LT}. \ Furthermore we define the inverse transform of $\mathbf{X}_{2}$,
i.e.
\begin{equation*}
\mathbf{Y}_{2}(t,\cdot )=\mathbf{X}_{2}(t,\cdot )^{-1}.
\end{equation*}%
It is easy to see
\begin{equation*}
\mathbf{Y}_{2}(t,\mathbf{x}_{2})=\mathbf{q}_{0}+{\mathbb{Q}}_{2}^{T}(t)(%
\mathbf{x}_{2}-\mathbf{q}_{2}(t)),\qquad \mathbf{x}_{2}\in \partial S_{2}(t).
\end{equation*}%
Finally we define the transformation $\widetilde{\mathbf{X}}_{1}:\Omega
_{F}^{2}(t)\rightarrow \Omega _{F}^{1}(t)$ in the following way:
\begin{equation*}
\widetilde{\mathbf{X}}_{1}(t,\mathbf{x}_{2})=\mathbf{X}_{1}(t,\mathbf{Y}%
_{2}(t,\mathbf{x}_{2}))
\end{equation*}%
and let $\widetilde{\mathbf{X}}_{2}(t,\mathbf{x}_{1})$ be its inverse, i.e.
\begin{equation*}
\widetilde{\mathbf{X}}_{2}(t,\cdot )=\widetilde{\mathbf{X}}_{1}(t,\cdot
)^{-1}.
\end{equation*}%
In the neighborhoods of $S_{1}(t)$ and $S_{2}(t)$ \ the transformations $%
\widetilde{\mathbf{X}}_{2}$ \ and $\widetilde{\mathbf{X}}_{1}$ are rigid.
They are given with the following expressions:%
\begin{equation}
\left\{
\begin{array}{c}
\widetilde{\mathbf{X}}_{1}(t,\mathbf{x}_{2})=\mathbf{q}_{1}(t)+{\mathbb{Q}}%
_{1}(t){\mathbb{Q}}_{2}^{T}(t)(\mathbf{x}_{2}-\mathbf{q}_{2}(t))\quad
\mathrm{in\;the\;neighborhood\;of\;}S_{2}(t), \\
\\
\widetilde{\mathbf{X}}_{2}(t,\mathbf{x}_{1})=\mathbf{q}_{2}(t)+{\mathbb{Q}}%
_{2}(t){\mathbb{Q}}_{1}^{T}(t)(\mathbf{x}_{1}-\mathbf{q}_{1}(t))\quad
\mathrm{in\;the\;neighborhood\;of\;}S_{1}(t).%
\end{array}%
\right.  \label{RigidTransf}
\end{equation}%
Furthermore we put ${\mathbb{Q}}={\mathbb{Q}}_{2}{\mathbb{Q}}_{1}^{T}$. \
Now we define a transformed solution of the strong solution $(\mathbf{u}%
_{2},p_{2},\mathbf{a}_{2},\boldsymbol{\omega }_{2}),$ \ where $p_{2}$ is a
respective pressure (see Appendix \ref{Sec:LT}):%
\begin{equation}
\left\{
\begin{array}{l}
\mathbf{U}_{2}(t,\mathbf{x}_{1})={\mathcal{J}_{\widetilde{\mathbf{X}}_{1}(t,\widetilde{\mathbf{%
X}}_{2}(t,\mathbf{x}_{1}))}\mathbf{u}_{2}(t,\widetilde{\mathbf{X}}_{2}(t,\mathbf{x}%
_{1})),}\qquad P_{2}(t,\mathbf{x}_{1})=p_{2}(t,\widetilde{\mathbf{X}}_{2}(t,%
\mathbf{x}_{1})), \\
\\
\mathbf{A}_{2}(t)={\mathbb{Q}}^{T}(t)\mathbf{a}_{2}(t),\qquad \qquad \quad
\boldsymbol{\Omega }_{2}(t)={\mathbb{Q}}^{T}(t)\boldsymbol{\omega }_{2}(t)%
\end{array}%
\right. \label{eq:Transformed}
\end{equation}%
with $ \mathcal{J}_{\widetilde{\mathbf{X}}_{1}(t,\widetilde{\mathbf{%
X}}_{2}(t,\mathbf{x}_{1}))} =\frac{\partial \widetilde{\mathbf{X}}_{1_i}}{\partial \mathbf{x}_{2_j}}$.
Using formulas \eqref{RigidTransf} we see that the following equalities hold
(see also \cite{GGH}):
\begin{equation}
\mathbf{n}_{1}={\mathbb{Q}}^{T}\mathbf{n}_{2},\qquad \mathbb{T}(\mathbf{u}%
_{2},p_{2})\mathbf{n}_{2}={\mathbb{Q}}{\mathcal{T}}(\mathbf{U}_{2},P_{2})%
\mathbf{n}_{1},  \label{Geom}
\end{equation}%
{where the transformed stress tensor $\mathcal{T}$ is defined in Appendix.}
Let us compute how the slip boundary condition is transformed, using the
transformed solution \eqref{eq:Transformed}. Let $\mathbf{u}_{s}^{1}$, $%
\mathbf{u}_{s}^{2}$ be the velocity of the bodies $S_{1}(t)$ and $S_{2}(t)$,
respectively (see \eqref{comp}). We define the transformed rigid velocity:
\begin{equation*}
\mathbf{U}_{s}^{2}(t,\mathbf{x}_{1})={\mathbb{Q}}^{T}\mathbf{u}_{s}^{2}(t,%
\widetilde{\mathbf{X}}_{2}(t,\mathbf{x}_{1}))=\mathbf{A}_{2}(t)+\boldsymbol{\Omega }%
_{2}(t)\times (\mathbf{x}_{1}-\mathbf{q}_{1}(t)).
\end{equation*}%
We use \eqref{Geom} to verify that $\mathbf{U}_{2}$ satisfies the slip
boundary condition:%
\begin{equation*}
\left\{
\begin{array}{l}
\mathbf{U}_{2}(t,\mathbf{x}_{1})\cdot \mathbf{n}_{1}={\mathbb{Q}}^{T}\mathbf{%
u}_{2}(t,\mathbf{x}_{2})\cdot \mathbf{n}_{1}=\mathbf{u}_{2}(t,\mathbf{x}%
_{2})\cdot \mathbf{n}_{2} \\
\qquad \qquad \qquad =\mathbf{u}_{s}^{2}(t,\mathbf{x}_{2})\cdot \mathbf{n}%
_{2}=\mathbf{U}_{s}^{2}(t,\mathbf{x}_{1})\cdot \mathbf{n}_{1}\ \quad \quad
\mathrm{on}\;\partial S_{1}(t), \\
\\
\mathcal{T}(\mathbf{U}_{2},P_{2})\mathbf{n}_{1}\cdot \boldsymbol{\tau }_{1}={%
\mathbb{Q}}^{T}\mathbb{T}(\mathbf{u}_{2},p_{2})\mathbf{n}_{2}\cdot {\mathbb{Q%
}}^{T}\boldsymbol{\tau }_{2} \\
\qquad \qquad \qquad \,\quad =\beta (\mathbf{u}_{s}^{2}-\mathbf{u}_{2})\cdot
\boldsymbol{\tau }_{2}=\beta (\mathbf{U}_{s}^{2}-\mathbf{U}_{2})\cdot
\boldsymbol{\tau }_{1}\quad \quad \mathrm{on}\;\partial S_{1}(t).%
\end{array}%
\right.
\end{equation*}

Now we can prove the following lemma.

\begin{lemma}
\label{Strong1} The transformed solution $(\mathbf{U}_{2},P_{2},\mathbf{A}%
_{2},\boldsymbol{\Omega }_{2})$\ of the strong solution $(\mathbf{u}%
_{2},p_{2},\mathbf{a}_{2},\boldsymbol{\omega }_{2}),$\ defined by %
\eqref{eq:Transformed}, satisfy the following system of equations on the
fluid domain $\Omega _{F}^{1}(t)$:%
\begin{equation}
\left.
\begin{array}{l}
\partial _{t}\mathbf{U}_{2}+(\mathbf{U}_{2}\cdot \nabla )\mathbf{U}%
_{2}-{\mu}\triangle \mathbf{U}_{2}+\nabla P_{2}={\mu}(\mathcal{L}-\triangle )\mathbf{U}%
_{2}-\mathcal{M}\mathbf{U}_{2} \\
\qquad \qquad \qquad \qquad \qquad \qquad \qquad \qquad -\tilde{\mathcal{N}}%
\mathbf{U}_{2}-(G{-}\nabla )P_{2}, \\
\mbox {\rm div }\mathbf{U}_{2}=0%
\end{array}%
\right\} \;\mathrm{in}\;(0,T)\times \Omega _{F}^{1}(t),
\label{transformedNS}
\end{equation}%
\begin{equation}
\left.
\begin{array}{l}
(\mathbf{U}_{2}-\mathbf{U}_{s}^{2})\cdot \mathbf{n}_{1}=0, \\
\mathcal{T}(\mathbf{U}_{2},P_{2})\mathbf{n}_{1}\cdot \boldsymbol{\tau }%
_{1}=\beta (\mathbf{U}_{s}^{2}-\mathbf{U}_{2})\cdot \boldsymbol{\tau }_{1}%
\end{array}%
\right\} \;\mathrm{on}\;(0,T)\times\partial S_{1}(t),
\label{transformedSlip}
\end{equation}%
\begin{align}
\mathbf{A}_{2}^{\prime }& =-\widetilde{\boldsymbol{\omega }}\times \mathbf{A}%
_{2}-\int_{\partial S_{1}(t)}\mathcal{T}(\mathbf{U}_{2},P_{2})\mathbf{n}%
_{1}\,d\mathbf{\gamma }(\mathbf{x}_{1})\qquad \mathrm{in}\;(0,T),
\label{transformedNewt} \\
(J_{1}\boldsymbol{\Omega }_{2})^{\prime }& =-\widetilde{\boldsymbol{\omega }}%
\times (J_{1}\boldsymbol{\Omega }_{2})-\int_{\partial S_{1}(t)}\left\{ (%
\mathbf{x}_{1}-\mathbf{q}_{1}(t))\times {\mathcal{T}}(\mathbf{U}_{2},P_{2})%
\mathbf{n}_{1}\right\} ~d\mathbf{\gamma }(\mathbf{x}_{1})\qquad \mathrm{in}%
\;(0,T),  \label{transformedAng}
\end{align}%
where the matrix $J_{1}$ and the vector $\widetilde{\boldsymbol{\omega }}$
are defined by
\begin{equation}
J_{1}={\mathbb{Q}}^{T}J_{2}{\mathbb{Q}}\ \qquad \text{and}\ \qquad
\widetilde{\boldsymbol{\omega }}\times \mathbf{x}={\mathbb{Q}}^{T}{\mathbb{Q}%
}^{\prime }\mathbf{x},  \label{w}
\end{equation}
{ and the operators $\mathcal{L}$, $\tilde{\mathcal{N}}$, $\mathcal{M}$ and $\mathcal{G}$ are defined in the Appendix, \eqref{OpL}-\eqref{OpP}.}
\end{lemma}

\begin{proof}
Equations \eqref{transformedNS} and \eqref{transformedSlip} follow from the
standard calculations, we refer to Appendix \ref{Sec:LT} and the above
considerations. The interested reader can find complete details of these
calculations, for example, in the articles \cite{IW,GGH}. Let us prove %
\eqref{transformedNewt}. We calculate:
\begin{eqnarray*}
\mathbf{A}_{2}^{\prime } &=&({\mathbb{Q}}^{T}\mathbf{a}_{2})^{\prime }={%
\mathbb{Q}}^{T}\mathbf{a}_{2}^{\prime }+({\mathbb{Q}}^{T})^{\prime }\mathbf{a%
}_{2}=-\int_{\partial S_{1}(t)}\mathcal{T}(\mathbf{U}_{2},P_{2})\mathbf{n}%
_{1}\,d\mathbf{\gamma }-{\mathbb{Q}}^{T}{\mathbb{Q}}^{\prime }{\mathbb{Q}}%
^{T}\mathbf{a}_{2} \\
&=&-\int_{\partial S_{1}(t)}\mathcal{T}(\mathbf{U}_{2},P_{2})\mathbf{n}_{1}\,d%
\mathbf{\gamma }-\widetilde{\boldsymbol{\omega }}\times \mathbf{A}_{2}.
\end{eqnarray*}

Let us now deduce \eqref{transformedAng}. We have:
\begin{eqnarray*}
{\mathbb{Q}}^{T}(J_{2}\boldsymbol{\omega }_{2})^{\prime } &=&-{\mathbb{Q}}%
^{T}\int_{\partial S_{2}(t)}\left\{ (\mathbf{x}_{2}-\mathbf{q}_{2}(t))\times
{\mathcal{T}}(\mathbf{u}_{2},p_{2})\mathbf{n}_{2}\right\} \,d\mathbf{\gamma }(%
\mathbf{x}_{2}) \\
&=&-\int_{\partial S_{1}(t)}\left\{ (\mathbf{x}_{1}-\mathbf{q}_{1}(t))\times
{\mathcal{T}}(\mathbf{U}_{2},P_{2})\mathbf{n}_{1}\right\} \,d\mathbf{\gamma }(%
\mathbf{x}_{1})
\end{eqnarray*}%
and on the other hand:
\begin{eqnarray*}
{\mathbb{Q}}^{T}(J_{2}\boldsymbol{\omega }_{2})^{\prime } &=&{\mathbb{Q}}%
^{T}({\mathbb{Q}}J_{1}{\mathbb{Q}}^{T}\boldsymbol{\omega }_{2})^{\prime } \\
&=&{\mathbb{Q}}^{T}{\mathbb{Q}}^{\prime }J_{1}\boldsymbol{\Omega }_{2}+(J_{1}%
\boldsymbol{\Omega }_{2})^{\prime }=\widetilde{\boldsymbol{\omega }}\times
(J_{1}\boldsymbol{\Omega }_{2})+(J_{1}\boldsymbol{\Omega }_{2})^{\prime }.
\end{eqnarray*}%
Therefore combining these last two relations we derive \eqref{transformedAng}%
. Hence this lemma is proven.$\hfill \;\blacksquare $
\end{proof}

\bigskip

As a consequence of the previous Lemma \ref{Strong1} we can give the weak
formulation for $(\mathbf{U}_{2},P_{2},\mathbf{A}_{2},\boldsymbol{\Omega }%
_{2}).$

\begin{corollary}
\label{weak2} Let us denote by
\begin{equation*}
\mathbf{F}={
\left \{
\begin{array}{lr}
\mu(\mathcal{L}-\triangle )\mathbf{U}_{2}-\mathcal{M}\mathbf{U}_{2}-%
\tilde{\mathcal{N}}\mathbf{U}_{2}-(\mathcal{G}{-}\nabla )P_{2}\; &{\rm in}\; (0,T)\times\Omega_F^1(t),
\\
0& {\rm in}\; (0,T)\times S_1(t).
\end{array}
\right .}
\end{equation*}
Then the transformed solution $\ (\mathbf{U}_{2},P_{2},\mathbf{A}_{2},%
\boldsymbol{\Omega }_{2})$ satisfies the following equality:%
\begin{eqnarray}
&&\int_{0}^{T}dt\int_{\Omega \backslash \partial S_{1}(t)}\mathbf{U}%
_{2}\cdot {\partial_t}\boldsymbol{\psi }~d\mathbf{x}_{1}+\int_{0}^{T}dt\int_{%
\Omega _{F}^{1}(t)}\Big ((\mathbf{u}_{1}\otimes \mathbf{U}_{2}):%
{\nabla^{T}\boldsymbol{\psi }}-(\mathbf{U}_{2}-\mathbf{u}_{1})\cdot \nabla \mathbf{U}%
_{2}\cdot \boldsymbol{\psi }\Big )\ \,d\mathbf{x}_{1}  \notag \\
&&-2\mu _{f}\int_{0}^{T}dt\int_{\Omega \backslash \partial S_{1}(t)}\big (%
\mathbb{D}\mathbf{U}_{2}:\mathbb{D}\boldsymbol{\psi }+\mathbf{F}\cdot
\boldsymbol{\psi }\big )\,d\mathbf{x}_{1}=\int_{0}^{T}dt\left\{
\int_{\partial S_{1}(t)}\beta (\mathbf{U}_{s}^{2}-\mathbf{U}_{2})\cdot (%
\boldsymbol{\psi }_{s}-\boldsymbol{\psi }_{f})\,d\mathbf{\gamma }(\mathbf{x}%
_{1})\right\}  \notag \\
&&-\int_{\Omega }\mathbf{u}_{0}\boldsymbol{\psi }(0,\cdot )\,d\mathbf{x}%
_{1}+\int_{0}^{T}(\widetilde{\boldsymbol{\omega }}\times (J_{1}\boldsymbol{%
\Omega }_{2})\cdot \boldsymbol{\psi }_{\omega }+\widetilde{\boldsymbol{%
\omega }}\times \mathbf{A}_{2}\cdot \boldsymbol{\psi }_{h})\ dt,
\label{Strong}
\end{eqnarray}%
which holds for any test function $\boldsymbol{\psi }$ \ satisfying %
\eqref{test}. Let us note that this function $\boldsymbol{\psi }$ is a rigid
one on $S_{1}(t)$, that is
\begin{equation*}
\boldsymbol{\psi }(t,\mathbf{x})=\boldsymbol{\psi }_{h}(t)+\boldsymbol{\psi }%
_{\omega }\times (\mathbf{x}-\mathbf{q}_{1}(t))\qquad \text{for }\mathbf{x}%
\in S_{1}(t).
\end{equation*}
\end{corollary}

\begin{proof}
Using the Reynolds transport theorem \eqref{eq6} we can write the inertial
term as:
\begin{eqnarray*}
&&\int_{0}^{T}dt\int_{\Omega _{F}^{1}(t)}(\partial _{t}\mathbf{U}_{2}\cdot
\boldsymbol{\psi }+\mathbf{U}_{2}\cdot \nabla \mathbf{U}_{2}\cdot
\boldsymbol{\psi })~\,d\mathbf{x}_{1} \\
&=&-\int_{0}^{T}dt\int_{\Omega _{F}^{1}(t)}\big (\mathbf{U}_{2}\cdot
\partial _{t}\boldsymbol{\psi }+\mathbf{u}_{1}\cdot \nabla (\mathbf{U}%
_{2}\cdot \boldsymbol{\psi })-\mathbf{U}_{2}\cdot \nabla \mathbf{U}_{2}\cdot
\boldsymbol{\psi }\big )\,d\mathbf{x}_{1}-\int_{\Omega }\mathbf{u}_{0}%
\boldsymbol{\psi }(0,\cdot )\,d\mathbf{x}_{1} \\
&=&-\int_{0}^{T}dt\int_{\Omega _{F}^{1}(t)}\big (\mathbf{U}_{2}\cdot
\partial _{t}\boldsymbol{\psi }+\mathbf{u}_{1}\otimes \mathbf{U}_{2}:{\nabla^T\boldsymbol{\psi }}+(\mathbf{u}_{1}-\mathbf{U}_{2})\cdot \nabla \mathbf{U}%
_{2}\cdot \boldsymbol{\psi }\big )\,d\mathbf{x}_{1}-\int_{\Omega }\mathbf{u}%
_{0}\boldsymbol{\psi }(0,\cdot )\,d\mathbf{x}_{1}.
\end{eqnarray*}%
The rest of the proof follows directly from Lemma \ref{Strong1} in a
classical way. The details can be found in Appendix A.1. of the article \cite%
{CN}.$\hfill \;\blacksquare $
\end{proof}

\bigskip

\bigskip

Before proceeding with the proof we need to prove the following two Lemmas
that give us estimate {for the additional terms arriving from transformation} in \eqref{transformedNS}-%
\eqref{transformedAng}.

{
\begin{lemma}
\label{EstimateRot} For the vector $\widetilde{\boldsymbol{\omega }}$
defined by \eqref{w} the equality holds:%
\begin{equation*}
\widetilde{\boldsymbol{\omega }}(t)=
\boldsymbol{\Omega }_{2}(t)-\boldsymbol{\omega }_{1}(t),\qquad t\in \lbrack 0,T_{0}].
\end{equation*}%
\end{lemma}
\begin{proof}
Let us denote by $\P_{\Omega_2}(t)$ matrix such that $\P_{\Omega_2}(t)\bx=\Omega_2(t)\times\bx$, $\bx\in\R^3$. Since $\Q(t)\in SO(3)$, $t\in (0,T)$, we have:
$$
\Omega_2\times\bx
=(\Q^T\bomega_2)\times\bx
=\Q^T(\bomega_2\times\Q\bx)
=\Q^T\P_2\Q\bx.
$$
Therefore,
\begin{equation}\label{POmega}
\P_{\Omega_2}(t)=\Q^T(t)\P_2(t)\Q(t)=\Q(t)\Q_2^{\prime}(t)\Q_1(t)^T.
\end{equation}
Here we used definitions of $\Q$ and $\P_2$, i.e. $\Q=\Q_2\Q_1^T$, $\P_2=\Q_2^{\prime}\Q_2^T$.
Now, we calculate $\Q^{\prime}$:
\begin{eqnarray*}
{\mathbb{Q}}^{\prime } &=&{\mathbb{Q}}_{2}^{\prime }{\mathbb{Q}}_{1}^{T}+{%
\mathbb{Q}}_{2}({\mathbb{Q}}_{1}^{T})^{\prime }={\mathbb{Q}}{\mathbb{Q}}^{T}{%
\mathbb{Q}}_{2}^{\prime }{\mathbb{Q}}_{1}^{T}-{\mathbb{Q}}_{2}{\mathbb{Q}}%
_{1}^{T}{\mathbb{Q}}_{1}^{\prime }{\mathbb{Q}}_{1}^{T} \\
&=&{\mathbb{Q}}({\mathbb{Q}}^{T}{\mathbb{Q}}_{2}^{\prime }-{\mathbb{Q}}%
_{1}^{\prime }){\mathbb{Q}}_{1}^{T}
=\Q(\P_{\Omega_2}-\P_1).
\end{eqnarray*}
Hence,
\begin{align*}
\P_{\tilde{\bomega}}=\Q^T\Q'=(\P_{\Omega_2}-\P_1),
\end{align*}
which finishes the proof of the Lemma. $\hfill \;\blacksquare $
\end{proof}
\begin{lemma}
\label{EstimateRHS} The following estimate holds:
\begin{eqnarray*}
\Vert {\mu}(\mathcal{L}-\triangle )\mathbf{U}_{2} &-&\mathcal{M}\mathbf{U}_{2}-%
\tilde{\mathcal{N}}\mathbf{U}_{2}-(\mathcal{G}{-}\nabla )P_{2}\Vert
_{L^{2}(0,T_{0};L^{2}(\Omega _{F}^{1}(t)))} \\
&\leq & C\left( ||\mathbf{a}_{1}-\mathbf{A}_{2}||_{L^{2}(0,T_{0})}+||%
\boldsymbol{\omega }_{1}-\boldsymbol{\Omega }_{2}||_{L^{2}(0,T_{0})} \right),
\end{eqnarray*}%
where {$C$ depends only on $\Vert \mathbf{U}_{2}\Vert
_{L^{2}(0,T_{0};H^{2}(\Omega _{F}^{1}(t)))}$, $\Vert P_{2}\Vert
_{L^{2}(0,T_{0};H^{1}(\Omega _{F}^{1}(t)))}$ and $\Vert \mathbf{U}_{2}\Vert
_{L^{\infty }(0,T_{0};H^{1}(\Omega _{F}(t)))}$.}\
\end{lemma}
{
\begin{proof}
First we estimate transformations $\widetilde{\mathbf{X}}_{2}$ and $%
\widetilde{\mathbf{X}}_{1}$. Since these transformations are rigid in the
neighborhood of the rigid body we have:
\begin{eqnarray*}
\widetilde{\mathbf{X}}_{2}(t,\mathbf{x}_{1}) &=&\mathbf{q}_{2}(t)+{\mathbb{Q}}(t)(%
\mathbf{x}_{1}-\mathbf{q}_{1}(t)).
\end{eqnarray*}%
We calculate:
\begin{align*}
\partial_t \widetilde{\mathbf{X}}_{2}(t,\mathbf{x}_{1})
&=\ba_2(t)+\Q^\prime(t)(\bx_1-\bq_1(t))-\Q(t)\ba_1(t)
\\
&=\Q(t)\big (\bA_2(t)-\ba_1(t)+\Q^T(t)\Q^\prime(t)(\bx_1-\bq_1(t))\big )
\\
&=\Q(t)\big ((\bA_2(t)-\ba_1(t))+(\bOmega_2(t)-{\bomega_1(t)})\times(\bx_1-\bq_1(t)\big ).
\end{align*}
In the last quality we used Lemma~\ref{EstimateRot}. By integration of the last equality we obtain:
\begin{align}\label{DtXEastimate}
(\widetilde{\mathbf{X}}_{2}-{\rm id})(t,\mathbf{x}_{1})
=\int_0^t\Q(r)\big ((\bA_2(r)-\ba_1(r))+(\bOmega_2(r)-{ \bomega_1(r)})\times(\bx_1-\bq_1(r)\big )dr.
\end{align}
Since $\Q$ and $\bq_1$ are uniformly bounded we get the following
estimates on $\partial S_{1}(t)$ for $t\in \lbrack 0,T_{0}]:$
\begin{eqnarray*}
|\widetilde{\mathbf{X}}_{2}(t)-\mathrm{id}|
&\leq & C(\Vert \mathbf{a}_{1}-\mathbf{A}_{2}\Vert _{L^{2}(0,T_{0})}+\Vert
\boldsymbol{\omega }_{1}-\boldsymbol{\Omega }_{2}\Vert _{L^{2}(0,T_{0})}),
\\
|\partial _{t}\widetilde{\mathbf{X}}_{2}(t)| &\leq& C(|\mathbf{a}_{1}(t)-%
\mathbf{A}_{2}(t)|+|\boldsymbol{\omega }_{1}(t)-\boldsymbol{\Omega }%
_{2}(t)|).
\end{eqnarray*}
Using the previous estimates and a standard construction of change of
variables connected to the rigid motion $\widetilde{\mathbf{X}}_{2}$ (see
the articles \cite{GGH,T} or for slightly different point of view we refer
to Proposition 1 and Corollary 1 of \cite{GS15}), one gets the following
estimates:
\begin{equation}
\left.
\begin{array}{l}
\Vert \widetilde{\mathbf{X}}_{2}(t,.)-\mathrm{id}\Vert _{W^{2,\infty
}(\Omega _{F}^{1}(t))}\leq C(\Vert \mathbf{a}_{1}-\mathbf{A}_{2}\Vert
_{L^{2}(0,T_{0})}+\Vert \boldsymbol{\omega }_{1}-\boldsymbol{\Omega }%
_{2}\Vert _{L^{2}(0,T_{0})}), \\
\quad \Vert \partial _{t}\widetilde{\mathbf{X}}_{2}(t,.)\Vert _{W^{1,\infty
}(\Omega _{F}^{1}(t))}\leq C(|\mathbf{a}_{1}(t)-\mathbf{A}_{2}(t)|+|%
\boldsymbol{\omega }_{1}(t)-\boldsymbol{\Omega }_{2}(t)|),%
\end{array}%
\right\} \qquad t\in \lbrack 0,T_{0}].  \label{TrEst1}
\end{equation}%
Analogous estimates can be derived for $\widetilde{\mathbf{X}}_{1}$.

To finish the proof we use the formulas for the transformed differential
operators \eqref{OpL}-\eqref{OpP}. Estimates \eqref{TrEst1} imply:
\begin{eqnarray*}
\Vert g_{ij}(t)-\delta _{ij}\Vert _{W^{1,\infty }(\Omega _{F}(t))}&+&\Vert
g^{ij}(t)-\delta _{ij}\Vert _{W^{1,\infty }(\Omega _{F}(t))}+\Vert \Gamma
_{ij}^{k}(t)\Vert _{L^{\infty }(\Omega _{F}(t))} \\
&\leq &C(|\mathbf{a}_{1}-\mathbf{A}_{2}|_{L^{2}(0,T_{0})}+|\boldsymbol{%
\omega }_{1}-\boldsymbol{\Omega }_{2}|_{L^{2}(0,T_{0})}),\qquad t\in \lbrack
0,T_{0}].
\end{eqnarray*}%
The proof of this lemma follows from the fact that {\ }%
\begin{equation*}
{\mathbf{U}_{2}\in L^{2}(0,T_{0};H^{2}(\Omega _{F}^{1}(t)))\cap L^{\infty
}(0,T_{0};H^{1}(\Omega _{F}(t))),\qquad P_{2}\in L^{2}(0,T_{0};H^{1}(\Omega
_{F}^{1}(t))).}
\end{equation*}
$\hfill \;\blacksquare $
\end{proof}
}
}

\bigskip

Let us give a principal lemma of our article from which Theorem \ref%
{WeakStrong} follows.

\begin{lemma}
\label{weak_strong} We have:
\begin{equation*}
(\mathbf{u}_{1},p_{1},\mathbf{a}_{1},\boldsymbol{\omega }_{1})=(\mathbf{U}%
_{2},P_{2},\mathbf{A}_{2},\boldsymbol{\Omega }_{2})\qquad \text{and}\qquad (%
\mathbf{u}_{1},\mathbf{B}_{1})=(\mathbf{u}_{2},\mathbf{B}_{2}).
\end{equation*}
\end{lemma}

\begin{proof}
\noindent { We begin by presenting formal estimates which are valid with some additional integrability (e.g. $\bu_1\in L^8_tL^4_x$, \cite{Temam}). After, we show that the result is valid for any weak solution.}
First we subtract equality \eqref{Strong} for $(\mathbf{U}%
_{2},P_{2},\mathbf{A}_{2},\boldsymbol{\Omega }_{2})$ from equality  \eqref{Weak} for $(\mathbf{u}_{1},p_{1},\mathbf{a}_{1},\boldsymbol{\omega }%
_{2}).$ In the obtained identity for the difference:
\begin{equation*}
(\mathbf{u},p,\mathbf{a},\boldsymbol{\omega })=(\mathbf{u}_{1}-\mathbf{U}%
_{2},p_{1}-P_{2},\mathbf{a}_{1}-\mathbf{A}_{2},\boldsymbol{\omega }_{1}-%
\boldsymbol{\Omega }_{2})
\end{equation*}%
we can take the test function $\boldsymbol{\psi }(r)=\mathbf{u}%
(1-sgn_{+}^{\varepsilon }(r-t))$ for any fixed $t\in (0,T)$ and pass on $%
\varepsilon \rightarrow 0,$\ that gives the identity:%
{
\begin{eqnarray}
&& \int_{0}^{t}dr \, \int_{\Omega \setminus \partial S_{1}(r)}\frac{1}{2}%
\partial _{{r}}|\mathbf{u}({r})|^{2}\,d\mathbf{x}+\int_{0}^{t}dr\int_{\Omega^1_{F}(r)}\Big (\mathbf{u}_{1}\otimes \mathbf{u}:\nabla^T%
\mathbf{u}-\mathbf{u}%
\cdot \nabla \mathbf{U}_{2}\cdot \mathbf{u}\,\Big )d\mathbf{x}  \notag \\
&&-\int_{\Omega \setminus \partial S_{1}(t)}|\mathbf{u}(t)|^{2}\,d\mathbf{x}%
-2\mu
_{f}\int_{0}^{t}dr\int_{\Omega \backslash \partial S_{1}(r)}\big (|\mathbb{D}%
\mathbf{u}|^{2}-\mathbf{F}\cdot \mathbf{u}\big )\,d\mathbf{x}  \notag \\
&&=\int_{0}^{t}dr\int_{\partial S(r)}\beta |\mathbf{u}_{s}-\mathbf{u}|^{2}d%
\mathbf{\gamma }-\int_{0}^{t}(\widetilde{\boldsymbol{\omega }}\times (J_{1}%
\boldsymbol{\Omega }_{2})\cdot \boldsymbol{\omega }+\widetilde{\boldsymbol{%
\omega }}\times \mathbf{A}_{2}\cdot \mathbf{a})\ dr.  \label{WeakDif}
\end{eqnarray}%
}
The main difficulty in the study of this relation is to estimate the
difference of the convective terms.

Let us combine the convective terms with the fluid acceleration term, then
we have:
{
\begin{eqnarray*}
\int_{0}^{t}dr\int_{\Omega _{F}^{1}(r)}\Big (\frac{1}{2}\partial _{{r}}|%
\mathbf{u}({r})|^{2} &+&\mathbf{u}_{1}\otimes \mathbf{u}:\nabla \mathbf{u}-\mathbf{u}%
\cdot \nabla \mathbf{U}_{2}\cdot \mathbf{u}\Big )\,d\mathbf{x}-\int_{\Omega
_{F}^{1}(t)}|\mathbf{u}(t)|^{2}\,\,d\mathbf{x} \\
&=&-\frac{1}{2}\int_{\Omega _{F}^{1}(t)}|\mathbf{u}(t)|^{2}\,\,d\mathbf{x}%
-\int_{0}^{t}dr%
\int_{\Omega _{F}^{1}(r)}\mathbf{u}\cdot \nabla \mathbf{U}_{2}\cdot \mathbf{u%
}\ \,d\mathbf{x}.
\end{eqnarray*}%
}
By integration by parts the last term in the right hand side of this
identity is written as:
\begin{equation}
\int_{0}^{t}dr\int_{\Omega _{F}^{1}( {r} )}\mathbf{u}\cdot \nabla \mathbf{U}%
_{2}\cdot \mathbf{u}\ d\mathbf{x}=
-\int_{0}^{t}dr\int_{\Omega _{F}^{1}({r})}%
\mathbf{u}\cdot \nabla \mathbf{u}\cdot \mathbf{U}_{2}\ d\mathbf{x}%
+\int_{0}^{t}dr\int_{\partial {\Omega _{F}^{1}({r})}}(\mathbf{u}\cdot \mathbf{n%
})(\mathbf{u}\cdot \mathbf{U}_{2})\ d\mathbf{\gamma }.  \label{int}
\end{equation}%
The first term in the right-hand side \ of \eqref{int} can be estimated in
the standard way (see e.g. Temam \cite{Temam}) by using the interpolation:
\begin{eqnarray*}
|\int_{0}^{t}dr\int_{\Omega _{F}^{1}({r})}\mathbf{u}\cdot \nabla \mathbf{u}%
\cdot \mathbf{U}_{2}\ d\mathbf{x}| &\leq &\int_{0}^{t}\Vert \mathbf{u}\Vert
_{L^{4}}\Vert \nabla \mathbf{u}\Vert _{L^{2}}\Vert \mathbf{U}_{2}\Vert
_{L^{4}}\ dr \\
&\leq &C\int_{0}^{t}(\Vert \mathbf{u}\Vert _{L^{2}}^{1/4}\Vert \nabla
\mathbf{u}\Vert _{L^{2}}^{3/4}+\Vert \mathbf{u}\Vert _{L^{2}})\Vert \nabla
\mathbf{u}\Vert _{L^{2}}\Vert \mathbf{U}_{2}\Vert _{L^{4}}\ dr \\
&\leq &C\varepsilon \int_{0}^{t}\Vert \nabla \mathbf{u}\Vert _{L^{2}}^{2}\
dr+\frac{C}{\varepsilon }\int_{0}^{t}\Vert \mathbf{u}\Vert _{L^{2}}^{2}\big (%
\Vert \mathbf{U}_{2}\Vert _{L^{4}}^{2}+\Vert \mathbf{U}_{2}\Vert _{L^{4}}^{8}%
\big )\ dr.
\end{eqnarray*}
{Here we used the notation $\Vert \cdot \Vert _{L^{p}}=\Vert \cdot\Vert _{L^{p}}(\Omega _{F}^{1}(r))$ for $p=2$ or $4$.}

The second term in the right-hand side of \eqref{int} is estimated as
follows:
\begin{eqnarray*}
&&|\int_{0}^{t}dr \int_{\partial {\Omega _{F}^{1}}({r})}(\mathbf{u}\cdot
\mathbf{n})(\mathbf{u}\cdot \mathbf{U}_{2})\ d\mathbf{\gamma }| \\
&\leq &C\int_{0}^{t}(|\mathbf{a}(r)|+|\boldsymbol{\omega }(r)|)\Vert \mathbf{%
u}(r)\Vert _{L^{2}(\partial S_{1}(r))}\Vert \mathbf{U}_{2}(r)\Vert
_{L^{2}(\partial S_{1}(r))}\ dr \\
&\leq &C\int_{0}^{t}(|\mathbf{a}(r)|+|\boldsymbol{\omega }(r)|)\Vert \mathbb{%
D}\mathbf{u}(r)\Vert _{L^{2}(\Omega _{F}^{1}(r))}\Vert \mathbb{D}\mathbf{U}%
_{2}(r)\Vert _{L^{2}(\Omega _{F}^{1}(r))}\ dr \\
&\leq &\frac{C}{\varepsilon }\int_{0}^{t}\Vert \mathbf{u}(r)\Vert
_{L^{2}(\Omega )}^{2}\Vert \mathbb{D}\mathbf{U}_{2}(r)\Vert _{L^{2}(\Omega
_{F}^{1}(r))}^{2}\ dr+C\varepsilon \int_{0}^{t}\Vert \mathbb{D}\mathbf{u}%
(r)\Vert _{L^{2}(\Omega _{F}^{1}(r))}^{2}\ dr \\
&\leq &\frac{C}{\varepsilon }\int_{0}^{t}\Vert \mathbf{u}(r)\Vert
_{L^{2}(\Omega )}^{2}\ dr+C\varepsilon \int_{0}^{t}\Vert \mathbb{D}\mathbf{u}%
(r)\Vert _{L^{2}(\Omega _{F}^{1}(r))}^{2}\ dr.
\end{eqnarray*}%
Moreover, \ applying \eqref{eq6} of the Reynolds transport theorem we have:
\begin{equation*}
{\frac{1}{2}\frac{d}{dt}\int_{S_{1}(t)}\mathbf{|u|^{2}}\,d\mathbf{x}=\frac{1%
}{2}\int_{S_{1}(t)}\partial _{t}|\mathbf{u}|^{2}d\mathbf{x}+\int_{S_{1}(t)}%
\mathbf{u}_{1}\otimes \mathbf{u}:\nabla \mathbf{u}\,d\mathbf{x}.}
\end{equation*}

Finally, the remainder terms in (\ref{WeakDif}) can be estimated by using
Lemmas \ref{EstimateRHS} and \ref{EstimateRot}.

By putting these estimates together we conclude:
\begin{eqnarray*}
\Vert \mathbf{u}(t)\Vert _{L^{2}(\Omega )}^{2} &+&2{\mu}\int_{0}^{t}\Vert \mathbb{%
D}\mathbf{u} ({r}) \Vert^2 _{{L^{2}(\Omega
_{F}^{1}(r))}}\ dr\leq
C\varepsilon\int_{0}^{t} \Vert \mathbb{D}\mathbf{u}({r})\Vert^{2} _{{L^{2}(\Omega
_{F}^{1}(r))}}\ dr  \\
&+&C\int_{0}^{t}(|\mathbf{a}(r)|^{2}+|\boldsymbol{\omega }(r)|^{2})\
dr \\
&+&\frac{C}{\varepsilon }\int_{0}^{t}\Vert \mathbf{u}({r})\Vert _{L^{2}({ \Omega })}^{2}\big (\Vert \mathbf{U}_{2}(r)\Vert _{L^{4}(\Omega
_{F}^{1}(r))}^{2}+\Vert \mathbf{U}_{2}(r)\Vert _{L^{4}(\Omega
_{F}^{1}(r))}^{8}+1\big )\ dr \\
&+&C\int_{0}^{t}(|\mathbf{a}(r)|+|\boldsymbol{\omega }(r)|)\Vert \mathbf{u}%
(r)\Vert _{L^{2}({\Omega })}\ dr.
\end{eqnarray*}%
Using Young's inequality and taking {$\varepsilon=\frac{\mu}{C}$, the term ${ \Vert \mathbb{%
D}\mathbf{u} ({r}) \Vert^2 _{L^{2}(\Omega
_{F}^{1}(r))}
} $ } can be absorbed in the left-hand side and we get:
\begin{equation*}
\Vert \mathbf{u}(t)\Vert _{L^{2}(\Omega )}^{2}\leq C\int_{0}^{t}\Vert
\mathbf{u}(r)\Vert _{L^{2}(\Omega )}^{2}\big (1+\Vert \mathbf{U}_{2}(r)\Vert
_{L^{4}(\Omega _{F}^{1}(r))}^{2}+\Vert \mathbf{U}_{2}(r)\Vert _{L^{4}(\Omega
_{F}^{1}(r))}^{8})\ dr.
\end{equation*}%
Hence we finish the proof by applying the integral Gronwall's inequality and
conclude that $\mathbf{u}=0$.

{ \textbf{Justification of formal calculation}. Notice that $\bu$ does not have enough regularity to justify previous calculations. However, the formal calculation can be justified in the  standard way (e.g. see the proof of Theorem 4.2 in \cite{Galdi}, or the proof of Theorem III.3.9 in \cite{Temam}) using the fact that $\bU_2$ is a strong solution with the regularity stated in Theorem~\ref{theorem-strong}. Namely, $\bU_2$ can be taken as a test function in the weak formulation for the solution $\bu_1$. On the other hand, one can multiply \eqref{transformedNS} by $\bu_1$ and integrate by parts in just space variables. By combining these two equalities together with the energy equality for $\bU_2$ and the energy inequality for $\bu_1$ one can recover the presented formal calculation. We emphasize that in this argument we used the strong regularity of $\bU_2$, and therefore we did not need any regularization. $\hfill \;\blacksquare $
}
\end{proof}

\bigskip

\section{Appendix. Local transformation}

\label{Sec:LT}

\bigskip

Since the fluid domain depends on the motion of the rigid body, we transform
the problem to a fixed domain. We define the local transformation as in
Takahaski \cite{T}. Let us point that such type of transformation firstly
was suggested by Inoue and Wakimoto \cite{IW} and then extensively used in the context of
strong solution to fluid-rigid body systems (see e.g. \cite{GGH,T}). Here we just briefly
repeat the main facts about this transformation for the convenience of the reader. Let us just
emphasize that our case is slightly different since we are not transforming to the fixed
cylindrical domain, but form one moving domain to the other. However, the essential fact
for this transformation is that the change of variable is volume preserving diffeomorphism - which is true also on our case.

Let us $\delta (t)=dist(S(t),\partial \Omega )$. We fix $\delta _{0}$ such
that $\delta (t)>\delta _{0}$ and define the solenoidal velocity field $%
\Lambda (t,\mathbf{x})$ such that $\Lambda =0$ in the $\delta _{0}/4$
neighborhood of $\partial \Omega $, $\Lambda =\mathbf{a}(t)+\boldsymbol{%
\omega }(t)\times (\mathbf{x}-\mathbf{q}(t))$ in the $\delta _{0}/4$
neighborhood of $S(t)$. Let us define the flow $\mathbf{X}(t):\Omega
\rightarrow \Omega $ \ as the unique solution of the system:
\begin{equation*}
\frac{d}{dt}\mathbf{X}(t,\mathbf{y})=\Lambda (t,\mathbf{X}(t,{\mathbf{y}}%
)),\qquad \mathbf{X}(0,\mathbf{y})=\mathbf{y},\qquad \forall \ {\mathbf{y}}\in
\Omega .
\end{equation*}%
We denote $\mathbf{Y}$ the inverse of $\mathbf{X}$, \ i.e.
\begin{equation*}
\mathbf{Y}(t,\cdot )=\mathbf{X}(t,\cdot )^{-1}.
\end{equation*}

Let us write the unknown functions $({\mathbf{u}},p,\boldsymbol{\omega },%
\mathbf{a},\mathbb{T})$ by the change of variables $\mathbf{x}\rightarrow
\mathbf{y.}$ Then in the system of coordinates $\mathbf{y}\in \Omega $ we
obtain the new unknown functions:
\begin{equation*}
\left.
\begin{array}{l}
{\mathbf{U}(t,\mathbf{y})=\mathcal{J}_{Y}(t,\mathbf{X}(t,\mathbf{y}))\mathbf{%
u}(t,\mathbf{X}(t,\mathbf{y})),}\qquad P(t,\mathbf{y})=p(t,\mathbf{X}(t,%
\mathbf{y})), \\
\Xi (t)={\mathbb{Q}}^{t}(t)\boldsymbol{\omega }(t),\quad \quad \qquad \qquad
\qquad \qquad \xi (t)={\mathbb{Q}}^{t}(t)\mathbf{a}(t), \\
\mathcal{T}({\mathbf{U}}(t,\mathbf{y}),P(t,\mathbf{y}))=\mathbb{Q}^{T}(t)%
\mathbb{T}({\mathbb{Q}}(t){\mathbf{U}}(t,\mathbf{y}),P(t,\mathbf{y})){%
\mathbb{Q}}(t)%
\end{array}%
\right\} \;\mathrm{for}\;t\in \lbrack 0,T],\;\;\mathbf{y}\in \Omega _{0},
\end{equation*}%
The Jacobian of this change of variables $\mathbf{x}\rightarrow \mathbf{y}$
\ is denoted by
\begin{equation*}
{\mathcal{J}_{Y}(t,\mathbf{X}(t,\mathbf{y}))=\Big(\frac{\partial Y_{i}}{%
\partial {x}_{j}}\Big),}
\end{equation*}

In the sequel \ we derive a system which satisfies the new unknown functions
$({\mathbf{U}},P,\Xi ,\xi ,\mathcal{T}).$ Our system \eqref{NS}-\eqref{IC}
is written in terms of the variables $(t,\mathbf{x}),$ therefore \ we have
to rewrite these equations in terms of the new variables \ $\mathbf{Y}=%
\mathbf{Y}(t,\mathbf{x})$.

Let us first note that the determinant of the Jacobian $\mathcal{J}_{Y}$
equals to $1,$ since { $ \Lambda $ is a divergence free vector field}. Hence using these
change of variables we have:
\begin{eqnarray*}
\int_{\partial S(t)}{\mathbb{T}}(\mathbf{u},p)\mathbf{n}(t)\ d\mathbf{\gamma
}(\mathbf{x}) &=&{\mathbb{Q}}\int_{\partial S(0)}\mathcal{T}(\mathbf{U},P)%
\mathbf{N}\ {d\sigma }(\mathbf{y}), \\
\int_{\partial S(t)}(\mathbf{x}-\mathbf{q}(t))\times {\mathbb{T}}(\mathbf{u}%
,p)\mathbf{n}(t)\ d\mathbf{\gamma }(\mathbf{x}) &=&{\mathbb{Q}}%
\int_{\partial S(0)}\mathbf{y}\times \mathcal{T}(\mathbf{U},P)\mathbf{N}\ {%
d\sigma }(\mathbf{y}),
\end{eqnarray*}%
where ${d\sigma }$ indicates the surface measure over non-moving surface $%
\partial S(0)$.

In the sequel we derive the equations which satisfy these new unknown
functions $({\mathbf{U}},P,\Xi ,\xi ,\mathcal{T}).$ Let us introduce the
metric covariant tensor
\begin{equation*}
g_{ij}=X_{k,i}X_{k,j}, \qquad X_{k,i}=\frac{\partial X_{k}}{\partial y_{i}},
\end{equation*}%
the metric covariant tensor
\begin{equation*}
g^{ij}=Y_{i,k}Y_{j,k}, \qquad Y_{i,k}=\frac{\partial Y_{i}}{\partial x_{k}}
\end{equation*}%
and the Christoffel symbol (of the second kind)
\begin{equation*}
\Gamma _{ij}^{k}=\frac{1}{2}g^{kl}(g_{il,j}+g_{jl,i}-g_{ij,l}),\qquad
g_{il,j}=\frac{\partial {g_{il}}}{\partial y_{j}}.
\end{equation*}%
It is easy to observe that in particular it holds:
\begin{equation*}
\Gamma _{ij}^{k}=Y_{k,l}X_{l,ij}, \qquad X_{l,ij}=\frac{\partial X_{l}}{%
\partial y_{i}\partial y_{j}}.
\end{equation*}

Hence under the change of variables $\mathbf{x}\rightarrow \mathbf{y}$ \ the
operator $\mathcal{L}$ is the transformed Laplace operator and it is given
by
\begin{align}
(\mathcal{L}\mathbf{u})_{i} =& \sum_{j,k=1}^{n}\partial
_{j}(g^{jk}\partial_k \mathbf{u}_{i})+2\sum_{j,k,l=1}^{n}g^{kl}\Gamma
_{jk}^{i}\partial _{l}\mathbf{u}_{j}  \notag \\
& + \sum_{j,k,l=1}^{n}\big(\partial _{k}(g^{kl}\Gamma
_{jl}^{i})+\sum_{m=1}^{n}g^{kl}\Gamma _{jl}^{m}\Gamma _{km}^{i}\Big)\mathbf{u%
}_{j}.  \label{OpL}
\end{align}

The convection term is transformed into
\begin{equation}
(\mathcal{N}\mathbf{u})_i = \sum _{j=1}^n \mathbf{u}_j \partial _j \mathbf{u}%
_i + \sum_{j,k+1}^n \Gamma ^i_{jk} \mathbf{u}_j\mathbf{u}_k
=(\bu\cdot\nabla\bu)_i+(\tilde{\mathcal{N}}\bu)_i.  \label{OpC}
\end{equation}

The transformation of time derivative and gradient are given by
\begin{equation}
(\mathcal{M} \mathbf{u})_{i} = \sum _{j=1}^n \dot{\mathbf{Y}}_j \partial _j
\mathbf{u}_i + \sum _{j,k=1}^n \Big(\Gamma _{jk}^i \dot{\mathbf{Y}}_k +
(\partial _k \mathbf{Y}_i)(\partial _j \dot{\mathbf{X}}_k)\Big)\mathbf{u}_j.
\label{OpN}
\end{equation}

The gradient of pressure is transform as follows:
\begin{equation}
(\mathcal{G}p)_{i}=\sum_{j=1}^{n}g^{ij}\partial _{j}p.  \label{OpP}
\end{equation}%
Therefore combining all formulas \eqref{OpL}-\eqref{OpP} we see that after the change of variables the system %
\eqref{NS}-\eqref{IC} is transformed into the following system:
\begin{equation*}
\left.
\begin{array}{rcl}
{\partial_t}\mathbf{U}+(\mathcal{M}-{\mu}\mathcal{L})\mathbf{U}&=&-\mathcal{N}(\mathbf{U})-%
\mathcal{G}p, \\
\mbox {div }\mathbf{U}&=&0%
\end{array}%
\right\} \;\mathrm{in}\;{(0,T)\times\Omega _{F}(0)},
\end{equation*}%
\begin{equation*}
\left.
\begin{array}{rcl}
\frac{d}{dt}\boldsymbol{\xi }&=&-m(\boldsymbol{\Xi }\times {\boldsymbol{\xi }}%
)-\int_{\partial S(0)}\mathcal{T}(\mathbf{U},P)\mathbf{N}{d\sigma }, \\
I\frac{d}{dt}\boldsymbol{\Xi }&=&\boldsymbol{\Xi }\times (I\boldsymbol{\Xi }%
)-\int_{S(0)}\mathbf{y}\times \mathcal{T}(\mathbf{U},P)\mathbf{N}{d\sigma }%
\end{array}%
\right\} \;\mathrm{in}\;(0,T),
\end{equation*}%
\begin{equation*}
\begin{array}{rcl}
(\mathbf{U}-\mathbf{U}_{s})\cdot \mathbf{N}&=&0,\\  {\beta }(\mathbf{U}-%
\mathbf{U}_{s})\cdot \tau &=&-2{\mu}\mathbb{D}(\mathbf{U})\mathbf{N}\cdot \tau
\qquad \mathrm{on}\;(0,T)\times\partial S(0),
\end{array}
\end{equation*}%
\begin{eqnarray*}
\mathbf{U} &=&0\ \quad \text{on}\ \partial \Omega , \\
\boldsymbol{\xi }(0) &=&\mathbf{a}(0)\quad \mbox{ and }\quad \boldsymbol{\Xi
}(0)=\boldsymbol{\omega }(0),
\end{eqnarray*}%
where $\mathbf{U}_{s}=\boldsymbol{\Xi }(t)\times \mathbf{y}+\boldsymbol\xi (t)$ is
the transformed rigid velocity $\mathbf{u}_{s}$; $\ \mathbf{N}=\mathbf{N}($%
\textbf{$\by$}$)$ is the unit normal at \ $\mathbf{y}\in \partial S(0),$
directed inside of $S(0);$\ $I={\mathbb{Q}}^{t}\mathbb{J}{\mathbb{Q}}$ is
the transformed inertia tensor which no longer depends on time (see details
in the article \cite{GGH}).

%{\color {red} \textbf{Supplementary material}

%See supplementary material - paper  H. Al Baba, N.V. Chemetov, B. Muha and \v{S}. Ne\v{c}asov\'a,
%"Strong solutions in $L^{2}$  framework for
%fluid-rigid body interaction problem - mixed case conditions.}
\bigskip

\textbf{Acknowledgements:} \vskip0.25cm \textit{{ The authors would like to thank the anonymous referee for her/his comments that helped us to improve the manuscript. We also would like to thank A. Rado\v sevi\' c for discussion and her comments about the manuscript.} The work of \v{S}. Ne\v{c}%
asová was supported by Grant No. 16-03230S of GA\v{C} in the framework of
RVO 67985840. The work of B. Muha was supported by by Croatian Science Foundation grant number 9477.}

\end{document}